\documentclass[12pt, 14paper,reqno]{amsart}
\usepackage{amsmath,amsfonts,amssymb}

\theoremstyle{plain}
\newtheorem{theorem}{Theorem}
\newtheorem{lemma}{Lemma}

\theoremstyle{proof}
\theoremstyle{definition}

\theoremstyle{remark}
\newtheorem{remark}{Remark}
\theoremstyle{lamma}

\numberwithin{equation}{section}
\numberwithin{lemma}{section}
\numberwithin{theorem}{section}

\usepackage{amsmath}
\usepackage{amsfonts}   
\usepackage{amssymb}
\usepackage{amssymb, amsmath, amsthm}
\usepackage[breaklinks]{hyperref}

\theoremstyle{thmrm}

\newenvironment{dedication}
    {\vspace{3ex}\begin{quotation}\begin{center}\begin{em}}
    {\par\end{em}\end{center}\end{quotation}}

\begin{document}
\title[Pell-type equations]{Pell-type equations and class number of the maximal real subfield of a cyclotomic field}
\author{Azizul Hoque and Kalyan Chakraborty}
\address{Azizul Hoque @Azizul Hoque, Harish-Chandra Research Institute, HBNI,
Chhatnag Road, Jhunsi,  Allahabad 211 019, India.}
\email{ azizulhoque@hri.res.in}
\address{Kalyan Chakraborty @Kalyan Chakraborty, Harish-Chandra Research Institute, HBNI, Chhatnag Road, Jhunsi, Allahabad 211 019, India.}
\email{kalyan@hri.res.in}
\maketitle
\begin{dedication}
To Prof. S.D. Adhikari on his $60^{th}$ birthday with great respect and friendship
\end{dedication}
\begin{abstract}
We investigate the solvability of the Diophantine equation $x^2-my^2=\pm p$ in  integers for certain integer $m$ and prime $p$. Then we apply these results to produce family of maximal real subfield of a cyclotomic field whose class number is strictly bigger than $1$.
\end{abstract}
\section{Introduction}
 Let 
$$
f(x, y)= a_nx^n + a_{n-1}x^{n-1}y+.........+ a_1x y^{n-1} + 
a_0y^n
$$
be a homogeneous polynomial ($n\geq 3$).  In 1909,  Thue proved that for a non-zero integer $r$, the Diophantine equation $f(x, y) = r$ either has no solution or at most  finitely many integer solutions. In contrast to this situation, we consider a  degree $2$ homogeneous polynomial 
\begin{equation}\label{main}
x^2-my^2=r
\end{equation} 
where $m$ is a positive square-free integer. This equation
has attracted the attention of many researchers. In particular (\ref{main}) for $r=1$, 

\begin{equation}\label{pell}
x^2-my^2=1
\end{equation}
is the well known Pell equation and it has rich literature. 
Lagrange (see, [\cite{DI}, p. 358]) used the simple continued fraction expansion of $\sqrt{m}$ to show that (\ref{pell}) is solvable in integers $x$ and $y$ with $y\ne 0$. On the other hand, it is well-known that  
the negative Pell equation 
\begin{equation}\label{npell}
x^2-my^2=-1
\end{equation}
has integer solutions
if and only if $\mathit{l}(\sqrt{m})\equiv 1\pmod 2$, where $\mathit{l}(\sqrt{m})$ is the length of the period of the continued fraction of $\sqrt{m}$. 

 In 1768, Lagrange (see, [\cite{SE}, Oeuvres II, pp. 377-535]) developed a recursive method for solving (\ref{main}) with $gcd(x, y)=1$, thereby reducing the problem to the case where $|r|<\sqrt{m}$, in which the solutions $(x, y)$ in positive integers might be found among $(x_n, y_n)$, with $\frac{x_n}{y_n}$ a convergent of the simple continued fraction for $\sqrt{m}$. Later, he (see, [\cite{SE}, Oeuvres II, pp. 655-726]) developed another algorithm that considered as a generalization of the method of solving (\ref{pell}) and (\ref{npell}) using the simple continued fraction. This algorithm was further simplified by Matthews \cite{MA} which is again based on simple continued fractions.
Kaplan and Williams \cite{KW} gave  necessary and sufficient conditions for the solvability of (\ref{npell}) based on the solvability of (\ref{main}) when $r=-4$ in positive coprime integers. Sawilla et al. \cite{SSW} developed a method faster than that of Lagrange to deal with a general binary quadratic Diophantine equation.  

One of our goal is to discuss the solvability of (\ref{main}) for certain 
values of $m$ and $r$. More precisely,  we show for $r=\pm p$ that (\ref{main}) has no integer solutions if 
\begin{equation*}
m=
\begin{cases}
(2np)^2-1,\\
(2np)^2+3 \text{ with } 3|n,\\
((2n+1)p)^2+2,
\end{cases}
\end{equation*}
where $p$ is a prime and $n$ is a positive integer.
We also show that for $m=((2n+1)p)^2-2$ with $p$ a prime and $n$ a positive integer, $x^2-my^2=p$ has no integer solution. Further for $m=((2n+1)p)^2-2$ with $p$ an odd prime and $n$ a positive integer, $x^2-my^2=-p$ has only two solutions $(2,1)$ and $(5,2)$. These two solutions occur only when $p=3$. 

We now turn to a brief introduction of the non-triviality of class groups of the maximal real subfields of certain cyclotomic fields. Let $m$ be a positive integer and $\zeta_m$ be a primitive $m$-th root of unity. Then the field $K_m=\mathbb{Q}(\zeta_m + \zeta^{-1} _m)$ is the maximal real subfield of the cyclotomic field $\mathbb{Q}(\zeta_m)$. We denote by $\mathcal{H}(m)$ the class-number of the field $K_m$ and by $h(m)$ the class-number of the quadratic field $k_m= \mathbb{Q}(\sqrt{m})$. Ankeny et al. \cite{ACH} showed that $\mathcal{H}(p)>1$ when $p=(2n q)^2+1$ is a prime,  $q$ is also a prime and $n>1$ is an integer. Lang \cite{LA} also showed $\mathcal{H}(p)>1$ for any prime of the form $p=\lbrace (2n+1)q \rbrace ^2 + 4$, where $q$ is a prime and $n\geq 1$ is an integer. In 1987, Osada \cite{OS} generalized both the results. More precisely, he proved that $\mathcal{H}(m)>1$ if $m=(2n q)^2 +1$ is a square-free integer, where $q$ is a prime and $n$ is a positive integer such that $n\neq 1, q$. In the same paper he also proved that if $m=\lbrace (2n+1)q\rbrace ^2 +4$ is a square-free integer, where $q$ is a prime and $n$ is a positive integer such that $n\neq q$, then $\mathcal{H}(m)>1$. All these results have been obtained when $m\equiv 1 \pmod 4$.
On the other hand, Takeuchi \cite{TA} established  similar results for certain primes of the form $p\equiv 3 \pmod 4$. He proved that if both $12m+7$ and $p=\{3(8m+5)\}^2-2$ are primes, where $m\geq 0$ is an integer, then $\mathcal{H}(4p)>1$. He also proved that if both $12m+11$ and $p=\{3(8m+7)\}^2-2$ are primes, where $m\geq 0$ is an integer, then $\mathcal{H}(4p)>1$. Along the same line, Hoque and Saikia \cite{HS} proved that if $m=\{3(8g+5)\}^2-2$ is a square-free integer, where $g$ is a positive integer, then $\mathcal{H}(4m)>1$. They also showed that for any square-free integer $m=\lbrace 3(8g+7)\rbrace^2-2$, where $g$ is a positive integer, then $\mathcal{H}(4m)>1$. 

Our final goal here is to derive the non-triviality of class groups of the maximal real subfields of certain cyclotomic fields as an application of the results established in \S 2 regarding the non-solvability of the Pell-type equations. More precisely, we prove $\mathcal{H}(4m)>1$ for the following classes
\begin{equation*}
m=
\begin{cases}
(2np)^2-1 \text{ with } p\equiv 1\pmod 4,\\
(2np)^2+3 \text{ with } p\equiv \pm 1 \pmod 4,\\
((2n+1)p)^2+2 \text{ with } p\equiv \pm 1 \pmod 8,\\
((2n+1)p)^2-2 \text{ with } p\equiv 1, 3 \pmod 8,\\
\end{cases}
\end{equation*}
where $n$ is a positive integer and $p$ is a prime. To prove these results, we first show that $h(m)>1$ by using the precise results on Pell-type equations derived in \S 2, and then we use a result of Yamaguchi\cite{YA} which says that $h(m)$ is a divisor of $\mathcal{H}(4m)$. 

In the concluding section, we provide some examples illustrating results. We have used SAGE 6.2 for these computations.

\section{Pell-type equations}
We consider the Pell-type equation
\begin{equation}\label{eq2.1}
 x^2-my^2=\pm p.
 \end{equation} 
For convenience we treat (\ref{eq2.1}) separately in 
\begin{equation}\label{eq2.2}
  x^2-my^2=p
 \end{equation}
 and
 \begin{equation}\label{eq2.3}
  x^2-my^2=-p.
 \end{equation}
Throughout $p$ will be a prime, $m$ will be a positive square-free integer and $n \geq 1$ will be an integer unless otherwise specified.
\begin{theorem}\label{thm2.1}
The equation
$$
x^2 - my^2 =\pm p
$$
has no integer solutions for $m = (2np)^2 - 1$.
\end{theorem}
\begin{proof}
Let us suppose that (\ref{eq2.2}) has integer solution(s) and without loss of generality we assume that $(x_0, y_0)$ is the least possible of such solutions with $y_0\geq 1$. Thus
\begin{equation}\label{eq2.4}
x_0^2 - my_0^2 = p.
\end{equation}
In the norm form (\ref{eq2.4}) is 
\begin{equation}\label{eq2.5}
p = N(x_0-y_0\sqrt{m}).
\end{equation}
We now multiply (\ref{eq2.5}) with the norm of the fundamental unit 
$$
\epsilon = 2np + \sqrt{m}
$$ 
in the field $\mathbb{Q}(\sqrt{m})$ and obtain
$$
p = (2np x_0 - my_0)^2 - (x_0 - 2np y_0)^2m.
$$
Now  
$$
y_0\leq |x_0 - 2np y_0|
$$
by the minimality of $y_0$. If $y_0\leq x_0 - 2np y_0$, then $(2np + 1)y_0\leq x_0$ and  therefore (\ref{eq2.4}) implies that
$$
p\geq \{(2np+1)^2-m\}y_0^2 = (4np+2)y_0^2.
$$
This is a contradiction as $n\geq 1$.

Again if $y_0\leq 2np y_0- x_0$, then $x_0\leq (2np-1)y_0$ and therefore (\ref{eq2.4}) implies 
$$
p\leq \{(2np-1)^2-m\}y_0^2=-(4np-2)y_0^2.
$$
This again leads to a contradiction as $n\geq 1$. Therefore (\ref{eq2.2}) has no integer solutions.  

We now suppose that (\ref{eq2.3}) has solutions in integers and again we can assume that $x_0$ is an integer and $y_0\geq 1$ is the least positive integer such that 
\begin{equation}\label{eq2.6}
x_0^2 - my_0^2 = - p.
\end{equation}
In the norm form (\ref{eq2.6}) is
\begin{equation}\label{eq2.7}
- p = N(x_0-y_0\sqrt{m}).
\end{equation}
As before multiplying (\ref{eq2.7}) with the norm of the fundamental unit $$\epsilon = 2np + \sqrt{m}$$ in the field $\mathbb{Q}(\sqrt{m})$, we obtain
$$
- p = (2np x_0-m y_0)^2-(x_0-2npy_0)^2m.
$$
By the minimality of $y_0$, we have 
$$
y_0\leq |x_0-2npy_0|.
$$
If $y_0\leq x_0-2np y_0$, then $(2np+1)y_0\leq x_0$ and therefore (\ref{eq2.6}) implies 
$$
p\leq my_0^2-(2np+1)^2y_0^2 = -(4np+2)y_0^2.
$$
This is not possible.

Again if $y_0\leq 2np y_0- x_0$, then $x_0\leq (2np-1)y_0$ and thus (\ref{eq2.6}) implies 
$$p\geq \{m-(2np-1)^2\}y_0^2=4np y_0^2.$$
This leads to a contradiction as $n\geq 1$ and therefore (\ref{eq2.3}) has no integer solutions. 
%$(x, y)$ in integers.  
\end{proof}
%%%%%%%%%%%%%%%%%%%%%%%%%%%%%%%%%%%%%%%%%%%%%%%%%

\begin{theorem}\label{thm2.2}
The equation
$$
x^2 - my^2 =\pm p
$$
has no integer solutions for $m = (2np)^2 + 3$ when $n$  is a multiple of $3$.
\end{theorem}
\begin{proof}
Let  $d = 2np$ and so $m = d^2 + 3$.
Suppose that (\ref{eq2.2}) has an integer solution and without loss of generality we assume that $(x_0, y_0)$ is the least positive integer solution of (\ref{eq2.2}). Then we get (\ref{eq2.4}) which can be re-written as 
$$
x_0^2-(dy_0)^2=p+3y_0^2
$$ 
and therefore
$$
(x_0+dy_0)(x_0-dy_0)>0.
$$
Let $u=x_0 + d y_0$ and $v=x_0 - d y_0$. Then both $u$ and $v$ are positive  with $(u + 1)(v -1)\geq 0$ and hence 
\begin{equation}\label{eq2.8}
uv - 1\geq u - v .
\end{equation}
Furthermore $d = \frac{u-v}{2y_0}$ and $uv = x_0^2 - d y_0^2$ and this implies $p = uv - 3y_0^2$ (using (\ref{eq2.4})).

Now $d-p=\frac{1}{2y_0}(u-v-2y_0uv+6y_0^3)$ implies 
$$
u-v-2y_0uv+6y^3=2y_0(2n-1)p.
$$ 
Thus using (\ref{eq2.8}) we obtain 
$$
uv-1-2y_0uv
+6y_0^3\geq 2y_0(2n-1)p>0
$$ 
and hence 
$$
uv>\frac{6y_0^3-1}{2y_0-1}.
$$ 
Therefore 
\begin{equation}\label{eq2.9}
p=uv-3y_0^2<\frac{6y_0^3-1}{2y_0-1}-3y_0^2=\frac{3}{2}y_0+\frac{\frac{3}{2}y_0-1}{2y_0-1}<2y_0 .
\end{equation} 
Since $3|n$, so $\mathbb{Q}(\sqrt{m})$ is Richaud-Degert type, and thus the fundamental unit in this field is 
$$
\epsilon=\frac{(2d^2+3)+2d\sqrt{m}}{3}.
$$ 
We now multiply (\ref{eq2.5}) with the norm of $\epsilon$, and get
$$
p=\frac{1}{9}\{(2d^2+3)x_0-2dmy_0\}^2-\{2dx_0-(2d^2+3)y_0\}^2m.
$$
By the minimality of $y_0$, we have 
$$
y_0\leq \frac{1}{9}|2dx_0-(2d^2+3)y_0|.
$$
If $y_0\leq \frac{1}{9}\{2dx_0-(2d^2+3)y_0\}$, then $(d^2+6)y_0\leq dx_0$ and thus (\ref{eq2.4}) gives 
$$
d^2 p\geq (9d^2+36)y_0^2>9y_0^2d^2.
$$
This contradicts (\ref{eq2.9}). 

Again if $y_0\leq \frac{1}{9}\{(2d^2+3)y_0-2dx_0\}$, then $d x_0 \leq (d^2-3) y_0$. Therefore (\ref{eq2.4}) implies that
$$
d^2p\leq \{(d^2-3)^2-m\}y_0^2=-9(d^2-1)y_0^2<0.
$$
This is meaningless.

Let us assume that (\ref{eq2.3}) has an integer solution(s) and $(x_0, y_0)$ is the least positive integer solution of (\ref{eq2.3}). Then $y_0 \ne 0$ such that (\ref{eq2.6}) holds.

Suppose that $p > 3 y_0^2$. Then using (\ref{eq2.6}) we have
$$
(dy_0)^2-x_0^2=p-3y_0^2>0.
$$ 
That would give  $(d y_0 + x_0)(d y_0 - x_0) > 0$. 

Let $u = d y_0 + x_0$ and $v=d y_0- x_0$. Then both $u$ and $v$ are positive integers with $(u-1)(v-1)\geq 0$. Thus 
\begin{equation}\label{eq2.10}
uv + 1\geq u + v.
\end{equation}
Furthermore $d = \frac{u+v}{2y_0}$ and $uv=dy_0^2-x_0^2$ implies $p = uv + 3y_0^2$ (using (\ref{eq2.6})) and hence
$$
d-p=\frac{1}{2y_0}(u+v-2uv y_0- 6 y_0^3).
$$ 
This implies that
$$
u + v - 2uv y_0- 6y_0^3 = 2y_0 (2n-1)p.
$$ 
Therefore using (\ref{eq2.10}) we obtain 
$$
uv+1-2uvy_0 - 6y_0^3\geq 2y_0(2n-1)p>0
$$ 
and hence $uv(1-2y_0)+(1-6y_0^3)>0$.
This contradicts to the fact that $y_0\geq 1$.  

Now we deal with the case when $p < 3y_0^2$. 

Again we multiply (\ref{eq2.7}) with the norm of the fundamental unit 
$$
\epsilon=\frac{(2d^2+3)+2d\sqrt{m}}{3}
$$
in the field $\mathbb{Q}(\sqrt{m})$, and get
$$
-p = \frac{1}{9}\{(2d^2+3)x_0-2dmy_0\}^2-\{2dx_0-(2d^2+3)y_0\}^2m.
$$
By the minimality of $y_0$, we have 
$$
9y_0\leq |2dx_0-(2d^2+3)y_0|.
$$
If $9y_0\leq \{2dx_0-(2d^2+3)y_0\}$, then $(d^2+6)y_0\leq dx_0$ and thus using (\ref{eq2.6}), we have 
$$
-d^2p\geq (9d^2+36)y_0^2>9y_0^2d^2.
$$
This is meaningless. 

Finally,  if $9y_0\leq \{(2d^2+3)y_0-2dx_0\}$, then $dx_0\leq (d^2-3)y_0$ and therefore using (\ref{eq2.6}), we obtain 
$$
-d^2p\leq \{(d^2-3)^2-m\}y_0^2=-9(d^2-1)y_0^2.
$$
Thus $d^2p\geq (9d^2-9)y_0^2\geq 8d^2y_0^2$ as $d>3$ and hence $p\geq 8y_0^2$.
This contradicts to $p<3y_0^2$. 
\end{proof}
%%%%%%%%%%%%%%%%%%%%%%%%%%%%%%%%%%%%%%%%%%%%%%%%%%%%%%
\begin{theorem}\label{thm2.3}
Let $m=\{(2n+1)p\}^2\pm2$. 
Then 
\begin{itemize}
\item[(I)] the Diophantine equation
\begin{equation*}
x^2 - my^2 = p
\end{equation*}
has no integer solutions, and 
\item[(II)] the only positive integer solutions of
\begin{equation*}
x^2 - my^2 = -p
\end{equation*}
are $(2,1)$ and $(5,2)$. These two solutions occur only when $m=\{(2n+1)p\}^2-2$ with $p=3$.
\end{itemize}
\end{theorem}

\begin{proof} Let $d = (2n+1)p$.

(I)  Let us suppose that (\ref{eq2.2}) has solution $(x, y)$ in integers and that $(x_0, y_0)$ be the least positive integer solution of (\ref{eq2.2}). Then we obtain (\ref{eq2.4}) for $m = d^2\pm 2$. 
We consider the two cases separately:\\
Case-$(i):\ \ d^2+2,$ \\
Case-$(ii):\ \ d^2-2$.\\
Case-$(i)$: Let us consider $p>y_0$. Then (\ref{eq2.4}) implies 
$x_0^2-my_0^2>y_0$ and hence $x_0^2-(d y_0)^2>y_0+2y_0^2$. This gives 
$$
(x_0+dy_0)(x_0-dy_0)>0.
$$
Let  $u = x_0 + d y_0$ and $v = x_0 - d y_0$ and so both $u$ and $v$ are positive integers satisfying (\ref{eq2.8}).

Furthermore, $d= \frac{u-v}{2y_0}$ and $uv=x_0^2-d^2y_0^2$  give (using (\ref{eq2.4})) $p=uv-2y_0^2$. Therefore
$$
d-p =\frac{1}{2y_0}(u - v - 2uv y_0 + 4y_0^3).
$$ 
Thus
$$
u-v-2uv y_0 + 4y_0^3=4np y_0\geq 0.
$$
Hence using (\ref{eq2.8}), we obtain 
$$
\frac{4y_0^3-1}{2y_0-1}\geq uv.
$$
Therefore
$$
p=uv-2y_0^2\leq \frac{4y_0^3-1}{2y_0-1}-2y_0^2=y_0+\frac{y_0-1}{2y_0-1}
$$
and hence $p\leq y_0$ as $\frac{y_0-1}{2y_0-1}<1$. This contradicts our assumption. 

Now let $p\leq y_0$. As before we 
multiply (\ref{eq2.5}) with the norm of the fundamental unit
$$
\epsilon=(d^2+1)+d\sqrt{m}
$$ 
in field $\mathbb{Q}(\sqrt{m})$ and get,
$$
p=\{(d^2+1)x_0-dy_0m\}^2-\{x_0d-(d^2+1)y_0\}^2m.
$$
Since $y_0$ is the least value that satisfies (\ref{eq2.2}), we have 
$$y_0\leq |x_0d-(d^2+1)y_0|.$$ 
If $y_0\leq x_0d-(d^2+1)y_0$, then $my_0\leq dx_0$ and thus using (\ref{eq2.4}), we obtain 
$$
pd^2=(x_0d)^2-m(y_0d)^2\geq m(m-d^2)y_0^2=2my_0^2\geq 2mp^2
$$
as $p\leq y_0$. This implies $p\leq \frac{d^2}{2m}<\frac{1}{2}$ as $m = d^2+2$. This is not possible and hence in this case (\ref{eq2.2}) has no integer solutions. 

Case-$(ii)$: Now $p>2y_0^2$. Then using (\ref{eq2.4}) we see that 
$$(x_0+dy_0)(x_0-dy_0)=p-2y_0^2>0.$$
Let $u=x_0+ d y_0$ and $v= x_0 - d y_0$ and thus both $u$ and $v$ are positive integers  satisfying (\ref{eq2.8}). 

Furthermore $d=\frac{u-v}{2y_0}$ and $uv=x_0^2-d^2y_0^2$ which implies (using (\ref{eq2.4})) $p=uv+2y_0^2$. Therefore 
$$
d-p=\frac{1}{2y_0}(u - v- 2uvy_0-4y_0^3)
$$ 
which implies 
$$
u-v-2uv y_0- 4y_0^3 = 4np y_0\geq 0.
$$
Hence using (\ref{eq2.8}), 
$$
uv(1-2y_0)-(1+4y_0)\geq 0.
$$ 
This is a contradiction as $y_0>0$.

Let now $p<2y_0^2$. 
Then we multiply (\ref{eq2.5}) with the norm of the fundamental unit 
$$
\epsilon=(d^2-1)+d\sqrt{m}
$$ 
in the field $\mathbb{Q}(\sqrt{m})$, 
$$
p=\{(d^2-1)x_0-dm y_0\}^2-\{dx_0-(d^2-1)y_0\}^2m.
$$
By the minimality of $y_0$, we have 
$$
y_0\leq |dx_0-(d^2-1)y_0|.
$$
If $y_0\leq dx_0-(d^2-1)y_0$, then $d y_0\leq x_0$ and thus (\ref{eq2.4}) gives 
$$
p=x_0^2-my_0^2\geq d^2y_0^2-my_0^2=2y_0^2.
$$
This contradicts the assumption.

Again for $y_0\geq (d^2-1)y_0-dx_0$, we have
$dx_0\leq my_0$ and hence (\ref{eq2.4}) implies 
$$
d^2p=d^2x_0^2-my_0^2d^2 \leq m(m-d^2)y_0^2=-2my_0^2.
$$
This is not possible.   

(II) Let us assume that (\ref{eq2.3}) has solution in integers. Let $(x_0, y_0)$ be the least such positive integer solution. Then we obtain (\ref{eq2.6}) for $m=d^2\pm 2$. Again we handle the cases separately.\\
Case-$(a):\ \ d^2+2,$ \\
Case-$(b):\ \ d^2-2$.\\
Case-$(a)$:  If $p>2y_0^2$, then using (\ref{eq2.6}), we see that,
$(d y_0)^2-x_0^2>y_0+2y_0^2$ and hence 
$$
(x_0+dy_0)(x_0-dy_0)>0.
$$
If we set  $u=d y_0+x_0$ and $v=d y_0-x_0$ then both $u$ and $v$ are positive integers satisfying (\ref{eq2.10}). 

Additionally, $d=\frac{u+v}{2y_0}$ and $uv=-(x_0^2-d^2y_0^2)$ which implies (using (\ref{eq2.6})) that $p=uv+2y_0^2$. Therefore we have 
$$
d-p = \frac{1}{2y_0}(u+ v- 2uv y_0 - 4y_0^3).
$$ 
Thus
$$
u+v-2uv y_0-4y_0^3 = 4np y_0.
$$
Hence using (\ref{eq2.10}), we see that 
$$
4np y_0 \leq uv+1-2uvy_0-4y_0^3
=uv(1-2y_0)+(1-4y_0^3)< 0.
$$
This is not possible.
Therefore we must have $p<2y_0^2$.

We now multiply (\ref{eq2.7}) with the norm of the fundamental unit $$\epsilon=(d^2+1)+d\sqrt{m}$$ in field $\mathbb{Q}(\sqrt{m})$ and get,
$$
-p=\{(d^2+1)x_0-dy_0m\}^2-\{x_0d-(d^2+1)y_0\}^2m.
$$
By the minimality of $y_0$, we have 
$$
y_0\leq |x_0d-(d^2+1)y_0|.
$$ 
If $y_0\leq x_0d-(d^2+1)y_0$, then $my_0\leq dx_0$ and thus using (\ref{eq2.6}), we get 
$$
pd^2 = m(y_0d)^2-(x_0d)^2
\leq  m(d^2-m)y_0^2
=-2my_0^2<0.
$$
This is an impossibility.

Again if $y_0\leq (d^2+1)y_0-x_0d$, then $x_0\geq dy_0$ and hence using (\ref{eq2.6}), we obtain  
$$
p=my_0^2-x_0^2
\geq my_0^2-d^2y_0^2
=2y_0^2.
$$
This contradicts to $p<2y_0^2$. Therefore in this case (\ref{eq2.3}) has no integer solution.

Case-$(b)$:  If $p=3$ and $y_0=1$, then (\ref{eq2.6}) implies that $d=3, x_0= 2$ are the only possible values.
 Likewise, for $p=3, y_0=2$, (\ref{eq2.6}) shows that $d=3, x_0= 5$ are the only possible values. 

Further for $p=y_0=3$, (\ref{eq2.6}) gives $3d+x_0=7$ and $3d-x_0=3$ or $3d+x_0=21$ and $3d-x_0=1$. None of these cases are possible though. 

We now consider the case $p>3$ with $p\leq y_0$.

Again multiplying  (\ref{eq2.7}) with the norm of the fundamental unit 
$$
\epsilon=(d^2-1)+d\sqrt{m}
$$
in the field $\mathbb{Q}(\sqrt{m})$, we see that 
$$
-p=\{(d^2-1)x_0-dmy_0\}^2-\{dx_0-(d^2-1)y_0\}^2m.
$$
Since $y_0$ is the least value that satisfies (\ref{eq2.3}), we obtain
$$
y_0\leq |dx_0-(d^2-1)y_0|.
$$
If $y_0\leq dx_0-(d^2-1)y_0$, then $d y_0\leq x_0$ and hence (\ref{eq2.6}) gives 
$$
p = my_0^2-x_0^2
\leq (m-d^2)y_0^2
=-2y_0
< 0
$$
as $y_0>0$. 
This again gives us the required contradiction. 

Again if $y_0\leq (d^2-1)y_0-dx_0$, then $dx_0\leq my_0$ and hence (\ref{eq2.6}) implies 
$$
d^2p=md^2y_0^2-m^2y_0^2=2my_0^2\geq 2mp^2
$$ 
as $p\leq y_0$. This implies $d^2\geq 2mp=2p(d^2-2)$. This is a contradiction. 

The concluding case is  $p\geq 3$ with $p\geq y_0$. We use (\ref{eq2.6}) to get
$$
(dy_0+x_0)(dy_0-x_0)>0.
$$
Thus if we write $u=d y_0+x_0$ and $v=d y_0-x_0$, then both $u$ and $v$ are positive integers  satisfying (\ref{eq2.10}). 

Moreover $d=\frac{u+v}{2y_0}$ and $uv=-(x_0^2-d^2y_0^2)$ implies (using (\ref{eq2.6})) $p=uv-2y_0^2$. Thus 
$$
d-p=\frac{1}{2y_0}(u+v-2uv y_0+4y_0^3).
$$ 
This implies 
$$
4npy_0\leq uv+1-2uvy_0+4y_0^3
$$ 
and hence 
$$
uv\leq\frac{4y_0^3+1}{2y_0-1}.
$$ 
Therefore 
$$
p=uv-2y_0^2
\leq \frac{4y_0^3+1}{2y_0-1}-2y_0^3
=y_0+\frac{y_0+1}{2y_0-1}
\leq y_0
$$
as $y_0>3$. This contradicts to $p>y_0$. This complete the proof of Theorem \ref{thm2.3}.
\end{proof}
%\begin{remark}
%Equation (\ref{eq2.3}) has positive integer solutions only when $p=3$ and it has only two such solutions. 
%\end{remark}
%%%%%%%%%%%%%%%%%%%%%%%%%%%%%%%%%%%%%%%%%%%%%%%%%%%%%%%%%%%%
\section{Class number of the maximal real subfield of certain cyclotomic fields}
In this section we shall prove some results concerning the non-triviality of class groups of the maximal real subfields of certain cyclotomic fields by using the results established in \S 2. Throughout this section $m$ is a square-free integer. The following result of Yamaguchi \cite{YA} is one of the main ingredient in the proof of the rest of the results.
\begin{lemma}\label{yamaguchi}
If $\phi(m)>4$, where $\phi$ stands for the Euler function and $m>0$ is an integer, then $h(m)|\mathcal{H}(4m)$.
\end{lemma}

\begin{theorem}\label{thm3.1}
Let $p\equiv 1 \pmod 4$ and let $m=(2np)^2-1$ with $n\geq 1$ an integer. Then $\mathcal{H}(4m)>1$.
\end{theorem}
\begin{proof}
We see that $m=(2np)^2-1\equiv 3\pmod 4$.
Therefore (since $p\equiv 1\pmod 4$)
$$
\left(\frac{m}{p}\right)=\left(\frac{-1}{p}\right)=1.
$$ 
Thus  $p$ splits completely in the field $k_m$ as a product of a prime ideal $\mathfrak{P}\subseteq\mathcal{O}_{k_m}$ of degree one and its conjugate $\mathfrak{P}'$ with $N(\mathfrak{P})=p$. 

Our target is to show that $h(m) >1$ and then applying Lemma \ref{yamaguchi} we have the proof. Thus if possible let  $h(m)=1$. Then $\mathfrak{P}$ is principal in $\mathcal{O}_{k_m}$ and hence $\mathfrak{P}$ can be written as $\mathfrak{P}=(a+b\sqrt{m})$ for some $a, b\in\mathbb{Z}$.

Now $p=N(\mathfrak{P})=N(a+b\sqrt{m})$ which implies 
$$
p=|a^2-mb^2|,
$$ 
that is 
$$
a^2-mb^2=\pm p.
$$ 
This shows that $(a, b)$ is an integer solution of (\ref{eq2.1}) which contradicts Theorem \ref{thm2.1}. Therefore $h(m)>1$. 

Now $\phi(m)>4$ and hence by Lemma \ref{yamaguchi}, we complete the proof. 
\end{proof}
Along the same lines using Theorem \ref{thm2.2} and Lemma \ref{yamaguchi} we have:
\begin{theorem}\label{thm3.2}
Let $p\equiv \pm 1\pmod 4$ be a prime and let $m=(2np)^2+3$ with $n$ a positive integer and a multiple of $3$. Then $\mathcal{H}(4m)>1$.
\end{theorem}  
Finally
using Theorem \ref{thm2.3} and Lemma \ref{yamaguchi} we have:
\begin{theorem}\label{thm3.3}
If $p\equiv \pm 1\pmod 8$ is a prime and $n$ is a positive integer, then $\mathcal{H}(4\{((2n+1)p)^2+2\})>1$.
\end{theorem}
\begin{theorem}\label{thm3.4}
If $p\equiv 1, 3\pmod 8$ is a prime with $p\ne 3$ and $n$ is a positive integer, then $\mathcal{H}(4\{((2n+1)p)^2-2\})>1$.
\end{theorem}
\begin{remark}
The condition `square-free' on $m$ is not necessary to have non-trivial class number. For example, in Table 3, there is one $m$ (see, $*$ mark) with square part for which $\mathcal{H}(4m)>1$.
\end{remark}
\section{Numerical examples}
In this section we give some numerical examples corroborating  our results in \S 3. It is sufficient to compute the class numbers of each of the families of underlying real quadratic fields, i.e. $h(m)$'s. We compute these class numbers for  small values of $m$ and list them in the Tables below. 

\begin{table}[ht]
 \centering
\begin{tabular}{ | c  c  c c |c c c c|} 
 \hline
 $p$ & $n$ & $m=(2np)^2-1$ & $h(m)$ & $p$ & $n$ & $m=(2np)^2-1$ & $h(m)$\\
\hline
5&2&399&8&5&3&899&6\\\hline
5&4&1599&12&5&6&3599&10\\\hline
5& 7& 4899&16& 13&2&2703&12\\\hline
13&3&6083&8&13&4&10815&16\\\hline
13&5&16899&40&17&1&1155&8\\\hline
17&2&4623&16&17&3&10403&14\\\hline
17&4&18495&12&29&1&3363&8\\\hline
29&4&53823&40&37&3&49283&24\\\hline
37&6&197135&72&41&2&26895&32\\\hline
41&4&26895&32&41&5&168099&72\\\hline
53&1&11235&24&53&2&44943&20\\\hline
53&3&101123&36&59&2&55695&32\\
\hline
\end{tabular}
\caption{Numerical examples of Theorem \ref{thm3.1}. }
\end{table}

\begin{table}[ht]
 \centering
\begin{tabular}{ | c  c c c |c c c c|} 
 \hline
 $p$ & $n$ & $m=(2np)^2+3$ & $h(m)$ & $p$ & $n$ & $m=(2np)^2+3$ & $h(m)$\\
\hline
3&3&327&2&3&6&1299&8\\\hline
3&9&2919&8&5&3&903&4\\\hline
5&6&3603&4&5&9&8103&8\\\hline
7&3&1767&4&7&6&7059&8\\\hline
7&9&15879&12&11&3&4359&10\\\hline
11&6&174427&16&11&9&39207&16\\\hline
13&3&6087&10&13&6&24339&16\\\hline
13&9&54759&30&17&3&10407&6\\\hline
19&3&1299&16&19&6&51987&16\\\hline
19&9&116967&24&29&3&30279&18\\\hline
29&6&121107&24&29&9&272487&24\\\hline
31&3&34599&20&31&6&138387&24\\\hline
31&9&311367&36&37&3&49287&20\\\hline
37&6&197139&32&37&9&443559&78\\ \hline
41&3&60519&38&41&6&242067&36\\

 \hline
\end{tabular}
\caption{Numerical examples of Theorem \ref{thm3.2}. }
\end{table}

\begin{table}[ht]
 \centering
\begin{tabular}{ | c  c  c  c |c c c  c | } 
 \hline
 $p$ & $n$ & $m=\{(2n+1)p\}^2+2$ & $h(m)$ & $p$ & $n$ & $m=\{(2n+1)p\}^2+2$ & $h(m)$\\
\hline
7&1&443&3&7&2&1227&4\\\hline
17&1&2603&4&17&2&7227*&2\\\hline
17&3&14163&10&17&4&23411&10\\\hline
17&5&34971&18&23&1&4763&4\\\hline
23&2&13227&10&23&3&25923&16\\\hline
23&4&42851&20&23&5&64011&24\\\hline
31&1&9218&6&31&2&24027&10\\\hline
31&3&47091&32&31&4&77843&12\\\hline
31&5&116283&16&41&1&15131&15\\\hline
41&2&42027&10&41&3&82371&44\\\hline
41&4&136163&21&41&5&203403&24\\\hline
47&1&19883&6&47&2&55227&20\\\hline
47&4&178931&33&71&1&45371&22\\\hline
71&3&247011&44&71&4&408323&28\\\hline
71&5&609963&58&73&1&47963&9\\\hline
73&3&261123&38&73&4&431651&52\\

 \hline
\end{tabular}

\caption{Numerical examples of Theorem \ref{thm3.3}.}
\end{table}
\begin{table}[ht]
 \centering
\begin{tabular}{ | c  c  c  c |c c c c|} 
 \hline
 $p$ & $n$ & $m=\{(2n+1)p\}^2-2$ & $h(m)$ & $p$ & $n$ & $m=\{(2n+1)p\}^2-2$ & $h(m)$\\
\hline
11&1&1087&7&11&2&3023&3\\ \hline
11&3&5927&5&11&4&9799&18\\ \hline
11&5&14639&17&17&1&2599&14\\\hline
17&2&7223&4&17&3&14159&9\\ \hline
17&4&23407&16&17&5&34967&16\\\hline
19&1&3247&8&19&2&9023&8\\\hline
19&3&17687&6&19&4&29239&34\\\hline
41&1&15127&10&41&2&42023&15\\\hline
41&3&82367&14&41&4&136159&78\\\hline
43&1&16639&24&43&2&46223&16\\\hline
43&3&90599&19&43&4&149767&39\\\hline
43&5&223727&24&59&1&31327&27\\\hline
59&2&87023&12&59&3&170567&16\\\hline
59&4&281959&55&59&5&42119&66\\\hline
73&1&47959&42&73&2&133223&14\\\hline
73&3&261119&38&73&4&431647&46\\

 \hline
\end{tabular}
\caption{Numerical examples of Theorem \ref{thm3.4}. }
\end{table}
\vspace*{6cm}
\section*{Acknowledgement} The authors are indebted to the anonymous referee for his/her valuable suggestions which has helped improving the presentation of this manuscript.

\end{document}